\documentclass[a4paper,reqno]{amsart}

\usepackage{amsmath,amsfonts,amsthm,amssymb,latexsym,mathrsfs,enumerate}

\usepackage{hyperref}
\usepackage[all]{xy}

\newcommand{\diag}{\mathrm{diag}}
\newcommand{\tr}{\mathrm{tr}}
\newcommand{\ev}{\mathrm{ev}}
\newcommand{\id}{\mathrm{id}}

\newcommand{\cs}{\mathrm{C}^*}

\newcounter{claimcounter}
\newenvironment{claim}{
\smallskip\noindent
\refstepcounter{claimcounter}
\underline{Claim \theclaimcounter:}
}

\newenvironment{claimproof}{
\smallskip\noindent
\emph{Proof of Claim \theclaimcounter.}
}

\newtheoremstyle{smallcaps}
    {3pt}                    
    {3pt}                    
    {\itshape}                   
    {}                           
    {\sc}                   
    {.}                          
    {.5em}                       
    {}  
    
\newtheoremstyle{smallcapsdef}
    {3pt}                    
    {3pt}                    
    {}                   
    {}                           
    {\sc}                   
    {.}                          
    {.5em}                       
    {}  

\theoremstyle{smallcaps}

\newtheorem {thm}{Theorem}[section]

\newtheorem {prop}[thm]{Proposition}
\newtheorem {corollary}[thm]{Corollary}

\theoremstyle {smallcapsdef}

\newtheorem {remark}[thm]{Remark}

\numberwithin{equation}{section}

\begin{document}

\author{Bhishan Jacelon and Wilhelm Winter}
\date{\today}
\title{$\mathcal{Z}$ is universal}

\address{Mathematisches Institut\\Einsteinstr. 62\\48149 M\"unster, Germany} 
\email{b.jacelon@uni-muenster.de}
\email{wwinter@uni-muenster.de}
\thanks{Research supported by the DFG  through SFB 878, by the ERC through AdG 267079 and by the EPSRC through EP/G014019/1 and EP/I019227/1.}


\begin{abstract}
We use order zero maps to express the Jiang-Su algebra $\mathcal{Z}$ as a universal $\cs$-algebra on countably many generators and relations, and we show that a natural deformation of these relations yields the stably projectionless algebra $\mathcal{W}$ studied by Kishimoto, Kumjian and others. Our presentation is entirely explicit and involves only $^*$-polynomial and order relations.
\end{abstract}
\maketitle

\section{Introduction}

In Elliott's programme to classify simple, nuclear $\cs$-algebras using $K$-theoretic invariants, the Jiang-Su algebra $\mathcal{Z}$ plays a particularly prominent role (see \cite{Winter:2007qf}). While there are various ways of characterizing $\mathcal{Z}$ (see for example \cite{Jiang:1999hb} and \cite{Rordam:2009qy}), its most concise description (due to the second named author, in \cite{Winter:2009yq}) is as the unique initial object in the category of strongly self-absorbing $\cs$-algebras. Here, a separable, unital $\cs$-algebra $\mathcal{D} \ne \mathbb{C}$ is \emph{strongly self-absorbing} if there is an isomorphism $\varphi: \mathcal{D} \to \mathcal{D} \otimes \mathcal{D}$ that is approximately unitarily equivalent to the first factor embedding, cf.\ \cite{Toms:2007uq}. The statement that $\mathcal{Z}$ is an initial object in this category is equivalent to saying that every strongly self-absorbing $\cs$-algebra absorbs $\mathcal{Z}$ tensorially (i.e.\ is `$\mathcal{Z}$-stable').

Apart from $\mathcal{Z}$, the known strongly self-absorbing algebras are: the Cuntz algebras $\mathcal{O}_2$ and $\mathcal{O}_\infty$, UHF algebras of infinite type, and such UHF algebras tensored with $\mathcal{O}_\infty$. These all admit presentations as universal $\cs$-algebras (see Section~\ref{questions} for a discussion), and Theorem~\ref{main1} of this article provides such a description for $\mathcal{Z}$ which, although complicated, is explicit and algebraic in the sense that it involves only $^*$-polynomial and order relations. The proof relies on the `order zero' presentations of prime dimension drop algebras described in \cite{Rordam:2009qy} (see Section~\ref{prelim}), and gives a construction of $\mathcal{Z}$ as an inductive limit of such algebras with connecting maps defined in terms of generators and relations.

The Jiang-Su algebra may be thought of as a stably finite analogue of $\mathcal{O}_\infty$, and the $\cs$-algebra $\mathcal{W}$ constructed in \cite{Jacelon:2010fj} (and studied in another form in \cite{Kishimoto:1996yu}) has been similarly proposed as a stably finite analogue of $\mathcal{O}_2$. The conjecture that $\mathcal{W}\otimes \mathcal{W} \cong \mathcal{W}$, while still open, is known to have interesting consequences. For example, it is shown in \cite{Jacelon:2010fj} that among the $\cs$-algebras classified in \cite{Robert:2010qy}, those that are simple and have trivial $K$-theory would absorb $\mathcal{W}$ tensorially. On the other hand, L.\ Robert proves in \cite{Robert:2011fj} that the Cuntz semigroup of a $\mathcal{W}$-stable $\cs$-algebra is determined by the cone of its lower semicontinuous 2-quasitraces. These results indicate that $\mathcal{W}$ may play an important role in the classification of nuclear, stably projectionless $\cs$-algebras. In this article, we examine the structure of $\mathcal{W}$ rather than its role in classification, by showing in Theorem~\ref{main2} how to obtain $\mathcal{W}$ as a nonunital deformation of $\mathcal{Z}$.

The paper is organized as follows. In Section~\ref{prelim} we establish notation and recall various basic facts about order zero maps and dimension drop algebras. Section~\ref{z universal} contains the presentation of $\mathcal{Z}$ as a universal $\cs$-algebra (Theorems~\ref{main1}~and~\ref{alt1}), and Section~\ref{w universal} contains the corresponding description of $\mathcal{W}$ (Theorem~\ref{main2}). We conclude with some open questions in Section~\ref{questions}.

\subsection*{\sc Acknowledgements} The first author would like to thank Simon Wassermann, Stuart White, Rob Archbold and Ulrich Kr\"ahmer for carefully reading the version of this article that appeared in his doctoral thesis.

\section{Preliminaries} \label{prelim}
In this section, we collect some well-known facts about order zero maps and dimension drop algebras that are used throughout the article. (Detailed exposition of order zero maps can be found in \cite{Winter:2012pi} and \cite{Winter:2009sf}.) We denote by $e_{ij}$ (or $e_{ij}^{(n)}$) the canonical $(i,j)$-th matrix unit in $M_n = M_n(\mathbb{C})$.

Recall that a completely positive (c.p.)\ map $\varphi:B \to A$ has \emph{order zero} if it preserves orthogonality. Every completely positive and contractive (c.p.c.)\ order zero map $\varphi:B\to A$ (for $B$ unital) is of the form $\varphi(\cdot) = \pi_\varphi(\cdot)\varphi(1_B) = \varphi(1_B)\pi_\varphi(\cdot)$ for a $^*$-homomorphism $\pi_\varphi: B \to A^{**}$ called the \emph{supporting $^*$-homomorphism of $\varphi$}. We frequently use the notion of positive functional calculus provided by this decomposition: if $f\in C_0(0,1]$ is positive with $\|f\|\le 1$ then the map $f(\varphi):B \to A$ given by $f(\varphi)(\cdot) := \pi_\varphi(\cdot)f(\varphi(1_B))$ is a well-defined c.p.c.\ order zero map. It is easy to see that if $p\in B$ is a projection, then $f(\varphi)(p) = f(\varphi(p))$. On the other hand, if $\varphi(1_B)$ is a projection, then $\varphi$ is in fact a $^*$-homomorphism.

Finally, c.p.c.\ order zero maps $B \to A$ correspond bijectively to $^*$-homomorphisms $C_0((0,1],B) \to A$. For $B=M_n$, one way of interpreting this fact is to say that the cone $C_0((0,1],M_n)$ is \emph{the universal $\cs$-algebra generated by a c.p.c.\ order zero map on $M_n$}.  Equivalently, it is easy to check that $C_0((0,1],M_n)$ is the universal $\cs$-algebra on generators $x_1,\ldots,x_n$ subject to the relations $\mathcal{R}_n^{(0)}$ given by
\begin{equation} \label{cone}
\|x_i\|\le 1,\quad x_1\ge 0,\quad x_ix_i^*=x_1^2,\quad x_j^*x_j \perp x_i^*x_i\quad \text{for} \quad 1\le i\ne j \le n
\end{equation}
(for example by mapping $x_j$ to $t^{1/2}\otimes e_{1j}$, so that $t\otimes e_{ij}$ corresponds to $x_i^*x_j$). One can therefore view the statement
\begin{equation} \label{notation}
C_0((0,1],M_n) = \cs(\varphi \mid \text{$\varphi$ c.p.c.\ order zero on $M_n$})
\end{equation}
as an abbreviation for these relations.

\begin{remark} \label{M_2}
In the case $n=2$, $C_0((0,1],M_2)$ is the universal $\cs$-algebra $\cs(x\mid \|x\|\le 1, x^2=0)$. Therefore, if $A$ is a $\cs$-algebra and $v\in A$ is a contraction with $v^2=0$, then there is a unique c.p.c.\ order zero map $\psi:M_2 \to A$ with $\psi^{1/2}(e_{12}) = v$ (so that $\psi(e_{11}) = vv^*$ and $\psi(e_{22})=v^*v$).
\end{remark}

By a \emph{prime dimension drop algebra}, we mean a $\cs$-algebra of the form
\begin{equation}
Z_{p,q}:= \{f\in C([0,1], M_p\otimes M_q) \mid f(0) \in M_p \otimes 1_q, f(1) \in 1_p \otimes M_q\},
\end{equation}
where $p$ and $q$ are coprime natural numbers. The Jiang-Su algebra $\mathcal{Z}$ is the unique inductive limit of prime dimension drop algebras which is simple and has a unique tracial state (see \cite{Jiang:1999hb}).

The order zero notation (\ref{notation}) essentially appears in \cite[Proposition 2.5]{Rordam:2009qy}, where the presentation of prime dimension drop algebras described in  \cite[Proposition 7.3]{Jiang:1999hb} is reinterpreted in terms of order zero maps. Specifically, the prime dimension drop algebra $Z_{p,q}$ is the universal unital $\cs$-algebra
\[
\cs(\alpha, \beta \mid  \text{$\alpha$ c.p.c.\ order zero on $M_p$}, \text{$\beta$ c.p.c.\ order zero on $M_q$}, \alpha(1_p) + \beta(1_q)=1, [\alpha(M_p), \beta(M_q)]=0),
\]
with generators corresponding to the obvious embeddings of $C_0([0,1),M_p)$ and $C_0((0,1],M_q)$ into $Z_{p,q}$.

When $q=p+1$, there is another presentation of $Z_{p,p+1}$ in terms of order zero maps that does not involve a commutation relation. The following is essentially contained in \cite[Proposition 5.1]{Rordam:2009qy}, and we note that these relations have already proved highly useful, for example in \cite{Winter:2010hl}, \cite{Winter:2012pi}, \cite{Sato:2010rz} and \cite{Matui:2012qv}.

\begin{prop} \label{universal drop}
Let $Z^{(n)}$ be the universal unital $\cs$-algebra $\cs(\varphi, \psi \mid \mathcal{R}_n)$, where $\mathcal{R}_n$ denotes the set of relations:
\begin{enumerate}[(i)]
\item $\varphi$ and $\psi$ are c.p.c.\ order zero maps on $M_n$ and $M_2$ respectively;
\item $\psi(e_{11})=1-\varphi(1_n)$;
\item $\psi(e_{22})\varphi(e_{11})=\psi(e_{22})$.
\end{enumerate}
Then $Z^{(n)} \cong Z_{n,n+1}$.
\end{prop}

In Section~\ref{z universal}, we use Proposition~\ref{universal drop} to write $\mathcal{Z}$ as a limit of dimension drop algebras in a universal way.
We make analogous use of Proposition~\ref{nonunital drop}, a nonunital version of Proposition~\ref{universal drop}, to present $\mathcal{W}$.

\section{Generators and relations for the Jiang-Su algebra} \label{z universal}

In this section, we will construct an inductive system $(Z^{(q(k))}, \alpha_k)$, where $q(k)=p^{3^k}$ for some fixed $p\ge 2$ ($p=2$ will do) and $Z^{(q(k))} = \cs(\varphi_k, \psi_k \mid \mathcal{R}_{q(k)}) \cong Z_{q(k),q(k)+1}$ (as in Proposition~\ref{universal drop}), and we will check that the inductive limit is simple with a unique tracial state. It will then follow from the classification theorem of \cite{Jiang:1999hb} that $\mathcal{Z} \cong \varinjlim (Z^{(q(k))}, \alpha_k)$.

If this procedure is to provide an explicit presentation of $\mathcal{Z}$ as a universal $\cs$-algebra, we need to be able to describe the connecting maps $\alpha_k$ in terms of generators and relations. (This is perhaps the key difference between our construction and the original construction of $\mathcal{Z}$ as an inductive limit in \cite{Jiang:1999hb}.) In other words, for every $k\in \mathbb{N}$ we will find c.p.c.\ order zero maps $\hat\varphi_k:M_{q(k)} \to Z^{(q(k+1))}$ and $\hat\psi_k:M_2 \to Z^{(q(k+1))}$ that satisfy the relations $\mathcal{R}_{q(k)}$ of Proposition~\ref{universal drop}. By universality, we will then have unital connecting maps $\alpha_k: Z^{(q(k))} \to Z^{(q(k+1))}$ with $\alpha_k \circ \varphi_k = \hat\varphi_k$ and $\alpha_k \circ \psi_k = \hat\psi_k$.

Before giving the connecting maps, it is instructive to note that there are obvious choices for $\hat\varphi_k$ and $\hat\psi_k$. Since $q(k+1) = q(k)^3$, we can identify $M_{q(k+1)}$ with $M_{q(k)} \otimes M_{q(k)} \otimes M_{q(k)}$ (and $e_{11}^{(q(k+1))}$ with $e_{11}^{(q(k))}\otimes e_{11}^{(q(k))}\otimes e_{11}^{(q(k))}$). We could then set $\hat\varphi_k = \varphi_{k+1} \circ (\id_{M_{q(k)}} \otimes 1_{q(k)} \otimes 1_{q(k)})$ and $\hat\psi_k = \psi_{k+1}$; it is easy to see that these maps satisfy the relations $\mathcal{R}_{q(k)}$, but the corresponding inductive limit certainly would not be simple. The idea is therefore to define $\hat\varphi_k$ in such a way as to ensure that $[0,1]$ is chopped up into suitably small pieces under the induced $^*$-homomorphism $\alpha_k$; $\hat\psi_k^{1/2}(e_{12})$ will then be some partial-isometry-like element that facilitates the relations $\mathcal{R}_{q(k)}$.

One way of doing this is as follows. Define $\rho_k: M_{q(k)} \to M_{q(k+1)}$ by
\begin{equation} \label{rho}
\rho_k = (\id_{M_{q(k)}} \otimes 1_{q(k)-1} \otimes 1_{q(k)}) \oplus \left(\bigoplus_{i=1}^{q(k)} \frac{i}{q(k)}\left(\id_{M_{q(k)}} \otimes e_{q(k),q(k)} \otimes e_{ii}\right)\right).
\end{equation}
Note that $\rho_k$ is c.p.c.\ order zero, with supporting $^*$-homomorphism $\pi_{\rho_k} = \id_{M_{q(k)}} \otimes 1_{q(k)} \otimes 1_{q(k)}$. We may then define $\hat\varphi_k := \varphi_{k+1} \circ \rho_k$. For this to work, we need to be able to transport the defect $1-\varphi_{k+1}(\rho_k(1_{q(k)})) = (1-\varphi_{k+1}(1_{q(k+1)})) + \varphi_{k+1}(1_{q(k+1)}-\rho_k(1_{q(k)}))$ underneath $\varphi_{k+1}(\rho_k(e_{11}^{(q(k))}))$, and the basic idea is to do this in two steps.
\begin{enumerate}
\item[Step 1.] Use $\psi_{k+1}(e_{12})$ to transport the corner $\pi_{\psi_{k+1}}(e_{11})(1-\varphi_{k+1}(\rho_k(1_{q(k)})))\pi_{\psi_{k+1}}(e_{11})$ underneath $\pi_{\psi_{k+1}}(e_{22})\varphi_{k+1}(e_{11}^{(q(k+1))})\pi_{\psi_{k+1}}(e_{22}) \le \varphi_{k+1}(e_{11}^{(q(k+1))}) \le \varphi_{k+1}(\rho_k(e_{11}^{(q(k))}))$.
\item[Step 2.] Use a partial isometry $v_{k+1} \in M_{q(k+1)}$ to transport (a projection bigger than) $1_{q(k+1)}-\rho_k(1_{q(k)})$ underneath (a projection smaller than) $\rho_k(e_{11}^{(q(k))}) - e_{11}^{(q(k+1))}$, so that $\varphi_{k+1}(v_{k+1})$ transports the rest of $1-\varphi_{k+1}(\rho_k(1_{q(k)}))$ underneath $\varphi_{k+1}(\rho_k(e_{11}^{(q(k))}))- \varphi_{k+1}(e_{11}^{(q(k+1))})$.
\end{enumerate}

Although this is essentially the right idea, it needs fine-tuning in the guise of functional calculus. We achieve this in Theorem~\ref{alt1} by adjusting the relations for $Z^{(q(k))}$, while for Theorem~\ref{main1}, we modify $\hat\varphi_k$ and $\hat\psi_k$ using the following piecewise linear functions:
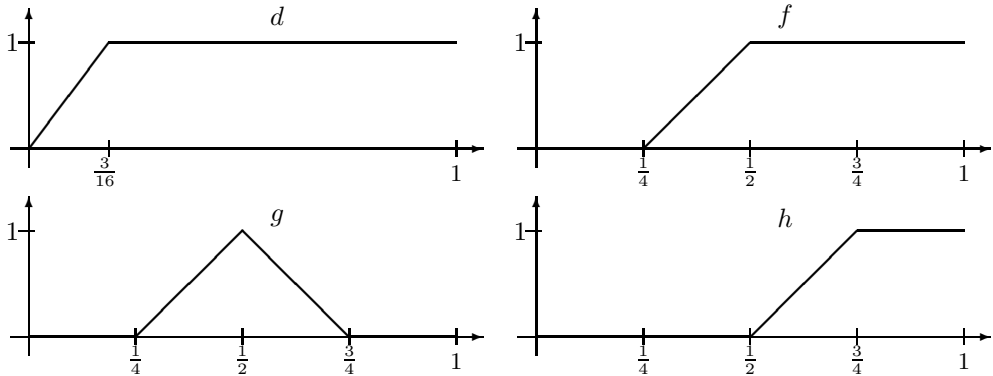
\begin{center}
\begin{picture}(367,70)
\put(0,10){\vector(1,0){177}}
\put(7,3){\vector(0,1){60}}
\put(37,7){\line(0,1){6}}
\put(167,7){\line(0,1){6}}
\put(3,50){\line(1,0){6}}
\thicklines
\put(7,10){\line(3,4){30}}
\put(37,50){\line(1,0){130}}
\put(1,50){\makebox(0,0){$1$}}
\put(37,-4){\makebox(0,0)[b]{\small $\frac{3}{16}$ \normalsize}}
\put(167,-2){\makebox(0,0)[b]{\small $1$\normalsize}}
\put(100,60){\makebox(0,0){$d$}}

\thinlines
\put(190,10){\vector(1,0){177}}
\put(197,3){\vector(0,1){60}}
\put(237,7){\line(0,1){6}}
\put(277,7){\line(0,1){6}}
\put(317,7){\line(0,1){6}}
\put(357,7){\line(0,1){6}}
\put(193,50){\line(1,0){6}}

\thicklines
\put(197,10){\line(1,0){40}}
\put(237,10){\line(1,1){40}}
\put(277,50){\line(1,0){80}}
\put(191,50){\makebox(0,0){$1$}}
\put(237,-4){\makebox(0,0)[b]{\small $\frac{1}{4}$\normalsize}}
\put(277,-4){\makebox(0,0)[b]{\small $\frac{1}{2}$\normalsize}}
\put(317,-4){\makebox(0,0)[b]{\small $\frac{3}{4}$\normalsize}}
\put(357,-2){\makebox(0,0)[b]{\small $1$\normalsize}}
\put(290,60){\makebox(0,0){$f$}}
\end{picture}

\begin{picture}(367,70)
\put(0,10){\vector(1,0){177}}
\put(7,3){\vector(0,1){60}}
\put(47,7){\line(0,1){6}}
\put(87,7){\line(0,1){6}}
\put(127,7){\line(0,1){6}}
\put(167,7){\line(0,1){6}}
\put(3,50){\line(1,0){6}}

\thicklines
\put(7,10){\line(1,0){40}}
\put(47,10){\line(1,1){40}}
\put(87,50){\line(1,-1){40}}
\put(127,10){\line(1,0){40}}
\put(1,50){\makebox(0,0){$1$}}
\put(47,-4){\makebox(0,0)[b]{\small $\frac{1}{4}$\normalsize}}
\put(87,-4){\makebox(0,0)[b]{\small $\frac{1}{2}$\normalsize}}
\put(127,-4){\makebox(0,0)[b]{\small $\frac{3}{4}$\normalsize}}
\put(167,-2){\makebox(0,0)[b]{\small $1$\normalsize}}
\put(100,55){\makebox(0,0){$g$}}

\thinlines
\put(190,10){\vector(1,0){177}}
\put(197,3){\vector(0,1){60}}
\put(237,7){\line(0,1){6}}
\put(277,7){\line(0,1){6}}
\put(317,7){\line(0,1){6}}
\put(357,7){\line(0,1){6}}
\put(193,50){\line(1,0){6}}

\thicklines
\put(197,10){\line(1,0){80}}
\put(277,10){\line(1,1){40}}
\put(317,50){\line(1,0){40}}
\put(191,50){\makebox(0,0){$1$}}
\put(237,-4){\makebox(0,0)[b]{\small $\frac{1}{4}$\normalsize}}
\put(277,-4){\makebox(0,0)[b]{\small $\frac{1}{2}$\normalsize}}
\put(317,-4){\makebox(0,0)[b]{\small $\frac{3}{4}$\normalsize}}
\put(357,-2){\makebox(0,0)[b]{\small $1$\normalsize}}
\put(290,55){\makebox(0,0){$h$}}
\end{picture}
\end{center}
These are chosen so that, writing $\bar d(t)=d(1-t)$, we have
\begin{equation} \label{dominance}
g= f-h, \quad hf=h, \quad (1-f)\bar d = 1-f \quad \text{and} \quad g\bar d=g.
\end{equation}
For use in Section~\ref{w universal}, we also note that if $\hat d$ is the function $\hat d(t)=d(t(1-t))$ then we have
\begin{equation} \label{w dominance}
(f-f^2)\hat d=f-f^2 \quad \text{and} \quad g\hat d=g.
\end{equation}
Finally, to accomplish Step 2, we choose a partial isometry
\begin{equation} \label{vdef}
v_{k+1}\in M_{q(k+1)}
\end{equation}
such that
\[
v_{k+1}v_{k+1}^* = 1_{q(k)} \otimes e_{q(k),q(k)} \otimes 1_{q(k)-1}
\] 
and
\[
v_{k+1}^*v_{k+1} = (e_{11} \otimes 1_{q(k)-1} \otimes 1_{q(k)}) + (e_{11} \otimes e_{q(k),q(k)} \otimes e_{q(k),q(k)}) - (e_{11} \otimes e_{11} \otimes e_{11}).
\]
This is possible since both of these projections have rank $q(k)^2-q(k)$; since they are orthogonal, we moreover have $v_{k+1}^2=0$. This $v_{k+1}$ then satisfies:
\begin{enumerate}[(i)]
\item \label{v1} $v_{k+1}^*v_{k+1} \perp e_{11}\otimes e_{11}\otimes e_{11} = e_{11}^{(q(k+1))}$ (in fact, $v_{k+1}v_{k+1}^*$ is orthogonal to $e_{11}^{(q(k+1))}$ too);
\item \label{v2} $v_{k+1}^*v_{k+1}$ is dominated by $\rho_k(e_{11}^{(q(k))})$ (and therefore by $\rho_k(e_{11}^{(q(k))})-e_{11}^{(q(k+1))}$); and
\item \label{v3} $v_{k+1}v_{k+1}^*$ acts like a unit on
\begin{equation} \label{1-rho}
1_{q(k+1)}-\rho_k(1_{q(k)}) = \bigoplus_{i=1}^{q(k)}\left(1-\frac{i}{q(k)}\right)(1_{q(k)} \otimes e_{q(k),q(k)} \otimes e_{ii}).
\end{equation}
\end{enumerate}

\begin{thm} \label{main1}
Let the functions $d,f,g,h \in C_0(0,1]$, the partial isometries $v_{k+1} \in M_{q(k+1)}$, and the c.p.c.\ order zero maps $\rho_k: M_{q(k)} \to M_{q(k+1)}$ be as above for each $k\in \mathbb{N}$. Define $\mathcal{Z}_U$ to be the universal unital $\cs$-algebra generated by c.p.c.\ order zero maps $\varphi_k$ on $M_{q(k)}$ ($k\in \mathbb{N}$) and $\psi_k$ on $M_2$ ($k\in \mathbb{N}$) such that for each $k$, these maps satisfy the relations $\mathcal{R}_{q(k)}$, i.e.\
\begin{equation} \label{block1}
\psi_k(e_{11})=1-\varphi_k(1_{q(k)})
\end{equation}
and
\begin{equation} \label{block2}
\psi_k(e_{22})\varphi_k(e_{11})=\psi_k(e_{22}),
\end{equation}
together with the additional relations $\mathcal{S}_{q(k)}$ given by
\begin{equation} \label{connect1}
\varphi_k = f(\varphi_{k+1}) \circ \rho_k,
\end{equation}
\begin{align} \label{connect2}
\psi^{1/2}_k(e_{12}) &= \left(1-f(\varphi_{k+1})(1_{q(k+1)}) + g(\varphi_{k+1})(1_{q(k+1)} - \rho_k(1_{q(k)}))\right)^{1/2} d(\psi_{k+1})(e_{12})\\
\notag &+ h(\varphi_{k+1})(1_{q(k+1)} - \rho_k(1_{q(k)}))^{1/2}f(\varphi_{k+1})(v_{k+1}).
\end{align}
Then $\mathcal{Z}_U\cong \mathcal{Z}$.
\end{thm}

\begin{proof}
For each $k$, define $\hat\varphi_k:M_{q(k)}\to Z^{(q(k+1))} = \cs(\varphi_{k+1}, \psi_{k+1} \mid \mathcal{R}_{q(k+1)})$ and $\hat\psi_k:M_2 \to Z^{(q(k+1))}$ by
\begin{equation}
\hat\varphi_k=f(\varphi_{k+1}) \circ \rho_k
\end{equation}
and
\begin{equation} \label{psihatdef}
\hat\psi^{1/2}_k(e_{12})= \gamma_k + \delta_k,
\end{equation}
where
\begin{equation}
\gamma_k:=\left(1-f(\varphi_{k+1})(1_{q(k+1)}) + g(\varphi_{k+1})(1_{q(k+1)} - \rho_k(1_{q(k)}))\right)^{1/2} d(\psi_{k+1})(e_{12})
\end{equation}
and
\begin{equation}
\delta_k:=h(\varphi_{k+1})\left(1_{q(k+1)} - \rho_k(1_{q(k)})\right)^{1/2}f(\varphi_{k+1})(v_{k+1}).
\end{equation}
We need to check that $\hat\varphi_k$ and $\hat\psi_k$ satisfy the relations $\mathcal{R}_{q(k)}$. First, it is obvious that $\hat\varphi_k$ is c.p.c.\ order zero since $\varphi_{k+1}$ and $\rho_k$ are, and $f$ is contractive. Next, to show that (\ref{psihatdef}) genuinely defines a c.p.c.\ order zero map $\hat\psi_k$, it suffices to check that $\gamma_k+\delta_k$ is a contraction that squares to zero (see Remark~\ref{M_2}). In fact, this would follow automatically from the relations (\ref{block1}) and (\ref{block2}) for $\hat\varphi_k$ and $\hat\psi_k$ (where, for the moment, we interpret $\hat\psi_k(e_{11})$ and $\hat\psi_k(e_{22})$ as \emph{notation} for $\hat\psi^{1/2}_k(e_{12})\hat\psi^{1/2}_k(e_{12})^*$ and $\hat\psi^{1/2}_k(e_{12})^*\hat\psi^{1/2}_k(e_{12})$ respectively). Indeed, $1-\hat\varphi_k(1_{q(k)})$ is certainly a contraction, and (\ref{block1}) and (\ref{block2}) would imply that
\begin{equation} \label{automatic}
\hat\psi_k(e_{22})\hat\psi_k(e_{11}) = \hat\psi_k(e_{22})(1-\hat\varphi_k(1_{q(k)})) = \hat\psi_k(e_{22}) - \sum_{i=1}^n \hat\psi_k(e_{22})\hat\varphi_k(e_{11})\hat\varphi_k(e_{ii}) = 0,
\end{equation}
and hence that $\left(\hat\psi^{1/2}_k(e_{12})\right)^2=0$. Let us now check that $\hat\varphi_k$ and $\hat\psi_k$ really do satisfy these relations.

\begin{claim} \label{block1 lemma}
$\hat\psi_k(e_{11})=1-\hat\varphi_k(1_{q(k)})$.
\end{claim}

\begin{claimproof}
First note that, using (\ref{block2}) and property (\ref{v1}) of the partial isometry $v_{k+1}$, we have
\[
d(\psi_{k+1})(e_{12})f(\varphi_{k+1})(v_{k+1}^*) = d^{1/2}(\psi_{k+1})(e_{12})d^{1/2}(\psi_{k+1})(e_{22})\varphi_{k+1}(e_{11})f(\varphi_{k+1})(v_{k+1}^*v_{k+1}v_{k+1}^*) = 0.
\]
Therefore, the cross terms $\gamma_k\delta_k^*$ and $\delta_k\gamma_k^*$ in the expansion of $\hat\psi_k(e_{11})=\hat\psi^{1/2}_k(e_{12})\hat\psi^{1/2}_k(e_{12})^*$ vanish.

Using the fact that $fh=h$, and property (\ref{v3}) of $v_{k+1}$, we have
\begin{align*}
h(\varphi_{k+1})(1_{q(k+1)} - \rho_k(1_{q(k)}))f(\varphi_{k+1})(v_{k+1})f(\varphi_{k+1})(v_{k+1}^*)&= h(\varphi_{k+1})((1_{q(k+1)} - \rho_k(1_{q(k)})v_{k+1}v_{k+1}^*)\\
&= h(\varphi_{k+1})(1_{q(k+1)} - \rho_k(1_{q(k)})).
\end{align*}
Thus, $\delta_k\delta_k^*= h(\varphi_{k+1})(1_{q(k+1)} - \rho_k(1_{q(k)}))$. From (\ref{block1}) we have
\[
d(\psi_{k+1})(e_{11}) = d(\psi_{k+1}(e_{11})) = d(1-\varphi_{k+1}(1_{q(k+1)})) = \bar d(\varphi_{k+1}(1_{q(k+1)})),
\]
where $\bar d(t) = d(1-t)$ as in (\ref{dominance}), whence we also obtain
\begin{align*}
(1-f(\varphi_{k+1})(1_{q(k+1)}))d(\psi_{k+1})(e_{11}) &= (1-f)(\varphi_{k+1}(1_{q(k+1)}))\bar d(\varphi_{k+1}(1_{q(k+1)}))\\
&= (1-f)(\varphi_{k+1}(1_{q(k+1)}))\\
&= 1-f(\varphi_{k+1})(1_{q(k+1)}).
\end{align*}
Similarly, we have $g(\varphi_{k+1})(1_{q(k+1)})d(\psi_{k+1})(e_{11}) = g(\varphi_{k+1})(1_{q(k+1)})$, hence
\[
g(\varphi_{k+1})(1_{q(k+1)} - \rho_k(1_{q(k)}))d(\psi_{k+1})(e_{11}) = g(\varphi_{k+1})(1_{q(k+1)} - \rho_k(1_{q(k)})).
\]
We therefore have $\gamma_k\gamma_k^*= 1-f(\varphi_{k+1})(1_{q(k+1)}) + g(\varphi_{k+1})(1_{q(k+1)} - \rho_k(1_{q(k)}))$. Since $g+h=f$, it follows that
\begin{align*}
\hat\psi_k(e_{11}) &= \gamma_k\gamma_k^* + \delta_k\delta_k^*\\
&= 1-f(\varphi_{k+1})(1_{q(k+1)}) + g(\varphi_{k+1})(1_{q(k+1)} - \rho_k(1_{q(k)})) + h(\varphi_{k+1})(1_{q(k+1)} - \rho_k(1_{q(k)}))\\
&= 1- f(\varphi_{k+1})(\rho_k(1_{q(k)}))\\
&= 1-\hat\varphi_k(1_{q(k)}).
\end{align*} 
\end{claimproof}

\begin{claim} \label{block2 lemma}
$\hat\psi_k(e_{22})\hat\varphi_k(e_{11})=\hat\psi_k(e_{22})$.
\end{claim}
\sloppy
\begin{claimproof}
Since $fh=h$ and $v_{k+1}$ is a partial isometry with property (\ref{v2}), we have
\begin{align*}
h(\varphi_{k+1})(1_{q(k+1)} - &\rho_k(1_{q(k)}))f(\varphi_{k+1})(v_{k+1})f(\varphi_{k+1})(\rho_k(e_{11}))\\
&= h(\varphi_{k+1})((1_{q(k+1)} - \rho_k(1_{q(k)}))v_{k+1})\\
&= h(\varphi_{k+1})(1_{q(k+1)} - \rho_k(1_{q(k)}))f(\varphi_{k+1})(v_{k+1}).
\end{align*}
Thus, $\delta_k\hat\varphi_k(e_{11})=\delta_k$. Next, it follows from (\ref{block2}), upon approximating $d^{1/2}$ and $f$ uniformly by polynomials, that 
\[
d^{1/2}(\psi_{k+1})(e_{22})f(\varphi_{k+1})(e_{11}) = f(1)d^{1/2}(\psi_{k+1})(e_{22}) = d^{1/2}(\psi_{k+1})(e_{22}).
\]
Since $e_{11}^{(q(k+1))} \perp (\rho_k(e_{11}^{(q(k))}) - e_{11}^{(q(k+1))})$ and $f(\varphi_{k+1})$ is order zero, we therefore have $d^{1/2}(\psi_{k+1})(e_{22})f(\varphi_{k+1})(\rho_k(e_{11})) = d^{1/2}(\psi_{k+1})(e_{22})$, hence $d(\psi_{k+1})(e_{12})f(\varphi_{k+1})(\rho_k(e_{11})) = d(\psi_{k+1})(e_{12})$. Therefore, $\gamma_k\hat\varphi_k(e_{11})=\gamma_k$, and so $\hat\psi_k(e_{22})\hat\varphi_k(e_{11})= (\gamma_k^* + \delta_k^*)(\gamma_k+\delta_k)\hat\varphi_k(e_{11}) = \hat\psi_k(e_{22})$.\\
\end{claimproof}
\fussy
We have now shown that $\hat\varphi_k$ and $\hat\psi_k$ satisfy the relations $\mathcal{R}_{q(k)}$. This means that, for any $k\in\mathbb{N}$, (\ref{connect1}) and (\ref{connect2}) do not introduce any new relations on $\varphi_{k+1}$ and $\psi_{k+1}$; thus, the sub-$\cs$-algebra generated by $\varphi_{k+1}$ and $\psi_{k+1}$ within $\mathcal{Z}_U$ is isomorphic to the universal $\cs$-algebra on relations $\mathcal{R}_{q(k+1)}$ (that is, to $Z^{(q(k+1))}$), and moreover contains the sub-$\cs$-algebra generated by $\varphi_{k}$ and $\psi_{k}$. Therefore, by Proposition~\ref{universal drop}, $\mathcal{Z}_U$ is isomorphic to an inductive limit of prime dimension drop algebras.

The strategy for the remainder of the proof is to pass from the abstract picture of $\mathcal{Z}_U$ as a universal $\cs$-algebra, to a concrete description as an inductive limit $\varinjlim (Z^{(q(k))}, \alpha_k)$, where the (unital) connecting maps $\alpha_k: Z^{(q(k))} \to Z^{(q(k+1))}$ are determined by (\ref{connect1}) and (\ref{connect2}) (i.e.\ $\alpha_k \circ \varphi_k = \hat\varphi_k$ and $\alpha_k \circ \psi_k = \hat\psi_k$). We will obtain explicit descriptions of the maps $\alpha_k$, and use these to show that $\mathcal{Z}_U$ is simple and has a unique tracial state.

For each $k\in \mathbb{N}$, let us fix an identification of $Z^{(q(k))}$ with $Z_{q(k),q(k)+1}$ via the order zero map $M_{q(k)} \to Z_{q(k),q(k)+1}$ (which, abusing notation, we also call $\varphi_k$) defined by:
\begin{equation} \label{phigen1}
\varphi_k(a)(t) = u_k(t)(a\otimes 1_{q(k)})u_k(t)^* \oplus (1-t)(a\otimes e_{q(k)+1,q(k)+1})
\end{equation}
for $a\in M_{q(k)}$ and $t\in [0,1]$. (Here, $u_k$ is a unitary in the algebra $C([0,1],M_{q(k)}\otimes M_{q(k)})$, included nonunitally in the top left corner of $C([0,1],M_{q(k)}\otimes M_{q(k)+1})$, with $u_k(0)=1$ and $u_k(1)$ implementing the flip in $M_{q(k)}\otimes M_{q(k)}$.) It is easy to write down a suitable $\psi_k$, but for the purpose of computing the connecting map $Z_{q(k),q(k)+1} \to Z_{q(k+1),q(k+1)+1}$ (also called $\alpha_k$), this is not necessary.

For each $t\in [0,1]$, let us write $\alpha_k^t$ for the map $\ev_t\circ\alpha_k: Z_{q(k),q(k)+1} \to M_{q(k+1)}\otimes M_{q(k+1)+1}$, where $\ev_t$ denotes evaluation at $t$. Then $\alpha_k^t$ is a finite-dimensional representation of $Z_{q(k),q(k)+1}$, so is a direct sum of finitely many irreducible representations $\pi_1^t,\ldots, \pi_{m(t)}^t$ of $Z_{q(k),q(k)+1}$ (corresponding up to unitary equivalence and, at the endpoints, up to multiplicity, to point evaluations). Since $\cs(\varphi_k(1_{q(k)}))\subset Z_{q(k),q(k)+1}$ separates the points of $[0,1]$, it is easy to see that the unitary equivalence classes of $\pi_1^t,\ldots, \pi_{m(t)}^t$ can be determined by computing $\alpha_k^t(\varphi_k(1_{q(k)}))$. To do this, note that
\begin{equation} \label{funphi1}
f(\varphi_{k+1})(b)(t) = u_{k+1}(t)(b\otimes 1_{q(k+1)})u_{k+1}(t)^* \oplus f(1-t)(b\otimes e_{q(k+1)+1,q(k+1)+1})
\end{equation}
for $b\in M_{q(k+1)}$, and recall the definition (\ref{rho}) of $\rho_k$. We then have, for $a\in M_{q(k)}$ and $t\in [0,1]$,
\begin{align*}
\alpha_k^t(\varphi_k(a)) &= f(\varphi_{k+1})(\rho_k(a))(t)\\
&= u_{k+1}(t)(a \otimes 1_{q(k)-1} \otimes 1_{q(k)} \otimes 1_{q(k+1)})u_{k+1}(t)^*\\
&\oplus u_{k+1}(t)\left(\bigoplus_{i=1}^{q(k)} \frac{i}{q(k)}\left(a \otimes e_{q(k),q(k)} \otimes e_{ii} \otimes 1_{q(k+1)} \right)\right)u_{k+1}(t)^*\\
&\oplus f(1-t)(a \otimes 1_{q(k)-1} \otimes 1_{q(k)} \otimes e_{q(k+1)+1,q(k+1)+1})\\
&\oplus f(1-t)\left(\bigoplus_{i=1}^{q(k)} \frac{i}{q(k)}\left(a \otimes e_{q(k),q(k)} \otimes e_{ii} \otimes e_{q(k+1)+1,q(k+1)+1}\right)\right)\\
&\sim_{\mathrm{u}} \left(\bigoplus_{m=1}^{q(k+1)}\bigoplus_{i=1}^{q(k)-1} \varphi_k(a)\left(1-\frac{i}{q(k)}\right)\right) \oplus \left(\bigoplus_{m=1}^{q(k)(q(k)-1)} \varphi_k(a)(1-f(1-t))\right)\\
&\oplus \left(\bigoplus_{i=1}^{q(k)} \varphi_k(a)\left(1-\frac{if(1-t)}{q(k)}\right)\right),
\end{align*}
where $\sim_{\mathrm{u}}$ denotes unitary equivalence. Write $h_i(t)= 1-\frac{if(1-t)}{q(k)}$ (so that, in fact, $h_{q(k)}=1-f(1-t)=h(t)$). By our earlier reasoning it then follows that, for every $t\in[0,1]$, there is a unitary $w_k(t)\in M_{q(k+1)}\otimes M_{q(k+1)+1}$  such that\begin{equation} \label{explicit alpha}
\alpha_k^t = w_k(t)\left(\left(\bigoplus_{m=1}^{q(k+1)} \bigoplus_{i=1}^{q(k)-1} \ev_{\frac{i}{q(k)}}\right) \oplus \left(\bigoplus_{m=1}^{q(k)(q(k)-1)} \ev_{h(t)}\right) \oplus \left(\bigoplus_{i=1}^{q(k)} \ev_{h_i(t)}\right)\right)w_k(t)^*.
\end{equation}
It could be that $t\mapsto w_k(t)$ is not continuous, but this does not matter. (Moreover, it is not difficult to show that, up to approximate unitary equivalence, continuity may be assumed anyway.)

We can also give a description of the connecting map $\alpha_{k,k+n} = \alpha_{k+n-1}\circ \cdots \circ \alpha_k$. For each $j\in \mathbb{N}$, let $\Lambda_j$ be the sequence of continuous functions given by listing each constant function $i/q(j)$ (for $1\le i \le q(j)-1$) with multiplicity $q(j+1)$, then $h$ with multiplicity $q(j)(q(j)-1)$ and then each $h_i$ for $1\le i \le q(j)$. Then $\alpha_{k,k+n}$ is fibrewise unitarily equivalent to the direct sum of all maps of the form $\ev_{F_1\circ \cdots \circ F_n}$ with $F_j \in \Lambda_{k+j-1}$ for $1\le j \le n$.

Let us write $T(A)$ for the space of tracial states on a $\cs$-algebra $A$. Recall that every tracial state on $Z_{q(j),q(j)+1}$ is of the form $\int\tr\circ\ev_t(\cdot) d\mu(t)$ for some Borel probability measure $\mu$ on $[0,1]$, where $\tr$ is the unique tracial state on $M_{q(j)}\otimes M_{q(j)+1}$. In particular, every such trace extends to a trace on $C([0,1],M_{q(j)}\otimes M_{q(j)+1})$, and is invariant under fibrewise unitary equivalence.

Since $\mathcal{Z}_U \cong \varinjlim Z_{q(k),q(k)+1}$ with unital connecting maps $\alpha_k$, we have $T(\mathcal{Z}_U) \cong \varprojlim T(Z_{q(k),q(k)+1})$. Thus $T(\mathcal{Z}_U)$ is an inverse limit of nonempty compact Hausdorff spaces, so is nonempty. That is, $\mathcal{Z}_U$ has at least one tracial state. For uniqueness, we need to show that for every $k\in \mathbb{N}$, every $\epsilon>0$, and every $b\in Z_{q(k),q(k)+1}$ we have
\begin{equation} \label{trace}
|\tau_1(\alpha_{k,k+n}(b)) - \tau_2(\alpha_{k,k+n}(b))| < \epsilon
\end{equation}
for all sufficiently large $n$ and every $\tau_1,\tau_2 \in T(Z_{q(k+n),q(k+n)+1})$. The key observation for this is that for each $j$, most of the elements in the sequence $\Lambda_j$ defined above are constant functions. In fact, the proportion of functions in $\Lambda_j$ that are \emph{not} constant is
\begin{equation} \label{proportion}
\frac{q(j)(q(j)-1)+q(j)}{q(j+1)(q(j)-1)+q(j)(q(j)-1)+q(j)} = \frac{q(j)^2}{q(j)^4-q(j)^3+q(j)^2} = \frac{1}{q(j)^2-q(j)+1}.
\end{equation}
Since $F_1\circ \cdots \circ F_n$ is constant if any of the $F_i$ are constant, it follows that for fixed $b\in Z_{q(k),q(k)+1}$, $\alpha_{k,k+n}(b)$ is fibrewise unitarily equivalent to a direct sum of continuous $M_{q(k)}\otimes M_{q(k)+1}$-valued functions, most of which are constant except for a small corner. But any two tracial states agree on the constant pieces, and the small corner has trace at most $\|b\| \prod_{j=k}^{k+n-1}\frac{1}{q(j)^2-q(j)+1}$, which of course converges to $0$ as $n \to \infty$. Thus (\ref{trace}) holds, and so $\mathcal{Z}_U$ has a unique tracial state.

It is well known that, to establish simplicity, it suffices to show the following (see for example \cite[Theorem 3.4]{Rordam:2009qy}): if $b$ is a nonzero element of $Z_{q(k),q(k)+1}$, then $\alpha_{k,r}(b)$ generates $Z_{q(r),q(r)+1}$ as a (closed, two-sided) ideal for every sufficiently large $r$ (which is the case if and only if $\alpha_{k,r}^t(b)$ is nonzero for every $t\in [0,1]$). Suppose that $b$ is such an element, so that there is an interval in $(0,1)$ of width $\epsilon>0$ on which $b$ is nonzero. For each $n\in \mathbb{N}$ and $t\in[0,1]$, $\alpha_{k,k+n+1}^t(b)$ contains summands unitarily equivalent to $b\left(h^{(n)}\left(\frac{i}{q(k+n)}\right)\right)$ for $1\le i \le q(k+n)-1$, where $h^{(n)}:=\overbrace{h\circ\cdots\circ h}^n$. Moreover, $h^{(n)}$ is of the form
\[
h^{(n)}(t)= \left\{
\begin{array}{ll}
0, & \quad 0\le t \le l_n/4^n\\
4^nt-l_n, & \quad l_n/4^n\le t \le (1+l_n)/4^n\\
1, & \quad (1+l_n)/4^n \le t \le 1
\end{array}
\right.
\]
for some $l_n$, and so it suffices to show that for large $n$ we have $\frac{1}{q(k+n)} < \frac{\epsilon}{4^n}$. But this is true for all large $n$ since $\frac{4^n}{q(k+n)} = \frac{4^n}{p^{3^{k+n}}} \longrightarrow 0 \quad \text{as} \quad n\to \infty$. Thus $\mathcal{Z}_U$ is simple.

It now follows from the classification theorem \cite[Theorem 6.2]{Jiang:1999hb} that $\mathcal{Z}_U\cong \mathcal{Z}$.
\end{proof}

\begin{remark} \label{algebraic}
One point that should be emphasized is that, despite the use of functional calculus, the relations of Theorem \ref{main1} really are \emph{algebraic}, or at least $\cs$-algebraic in the sense that they involve only $^*$-polynomial and order relations. This can be made explicit by encoding the relations (\ref{dominance}) satisfied by the functions $d$, $f$, $g$ and $h$ into the relations for the building blocks $Z^{(q(k))}$.

More specifically, it is not difficult to derive from Proposition~\ref{universal drop} that the dimension drop algebra $Z_{n,n+1}$ is isomorphic to the universal $\cs$-algebra on generators $\varphi$, $\psi$ and $h$ with relations:
\begin{enumerate}[(i)]
\item $\varphi$, $\psi$ and $h$ are c.p.c.\ order zero maps on $M_n$, $M_2$ and $\mathbb{C}$ respectively (in particular, $h$ is just a positive contraction);
\item $[\psi(e_{11}),\varphi(M_n)] = [h,\varphi(M_n)] =0$;
\item $\psi(e_{11})h = h$;
\item $h(1-\varphi(1_n))=1-\varphi(1_n)$;
\item $\psi(e_{22})\varphi(e_{11})=\psi(e_{22})$.
\end{enumerate}
(It is a straightforward exercise in functional calculus to write down inverse isomorphisms between the universal $\cs$-algebra determined by these relations and $Z^{(n)}\cong Z_{n,n+1}$.) The following is then proved in exactly the same way as Theorem~\ref{main1}.

\begin{thm} \label{alt1}
The Jiang-Su algebra $\mathcal{Z}$ is isomorphic to the universal unital $\cs$-algebra generated by c.p.c.\ order zero maps $\varphi_k$ on $M_{q(k)}$ ($k\in \mathbb{N}$) and $\psi_k$ on $M_2$ ($k\in \mathbb{N}$), and positive contractions $h_k$ ($k\in \mathbb{N}$), together with (for each $k\in \mathbb{N}$) the relations:
\begin{align*}
[\psi_k(e_{11}),\varphi_k(M_{q(k)})] &= [h_k,\varphi_k(M_{q(k)})] =0, \notag\\
\psi_k(e_{11})h_k &= h_k, \notag\\
h_k(1-\varphi_k(1_{q(k)})) &= 1-\varphi_k(1_{q(k)}), \notag\\
\psi_k(e_{22})\varphi_k(e_{11}) &= \psi_k(e_{22}), \notag\\
\varphi_k &= \varphi_{k+1} \circ \rho_k, \notag\\
\frac{1}{\sqrt2}(1+h_k)^{1/2}\psi_k^{1/2}(e_{12}) &= (h_{k+1} + (1-h_{k+1})\varphi_{k+1}(v_{k+1}v_{k+1}^*))^{1/2}\psi_{k+1}^{1/2}(e_{12}) \notag\\
&+ (1-\psi_{k+1}(e_{11}))^{1/2}\varphi_{k+1}^{1/2}(v_{k+1}),
\end{align*}
where the c.p.c.\ order zero maps $\rho_k: M_{q(k)} \to M_{q(k+1)}$ and the partial isometries $v_{k} \in M_{q(k)}$ are as in (\ref{rho}) and (\ref{vdef}) respectively. \qed
\end{thm}

\end{remark}

\section{$\mathcal{W}$ as a universal $\cs$-algebra} \label{w universal}

The article \cite{Razak:2002kq} (or, in a much more general setting, \cite{Robert:2010qy}) contains a classification by tracial data of simple inductive limits of building blocks
\begin{equation}
W_{n,m} := \{ f\in C([0,1], M_n\otimes M_{m}) \mid f(0) = a\otimes 1_{m}, f(1) = a\otimes 1_{m-1},  \: a\in M_n \}, \quad n,m \in \mathbb{N}, m> 1.
\end{equation}
Such building blocks are easily seen to be stably projectionless, and it can moreover be shown that they have trivial $K$-theory (this is why the classifying invariant is purely tracial). The classification is also complete in the sense that every permissible value of the invariant is attained---see \cite{Tsang:2005fj} or \cite[Proposition 5.3]{Jacelon:2010fj}. Then, $\mathcal{W}$ may be defined as the unique $\cs$-algebra in this class which has a unique tracial state (and no unbounded traces).

An explicit construction of $\mathcal{W}$ is given in \cite{Jacelon:2010fj}, and in this section we obtain another one by adapting the previous section's universal characterization of $\mathcal{Z}$. To begin with, notice that $W_{n,n+1}$ is isomorphic to a subalgebra of the dimension drop algebra $Z_{n,n+1}$; the following indicates that it in fact may be thought of as its nonuntial analogue (compare with Proposition~\ref{universal drop}).

\begin{prop} \label{nonunital drop}
Let $W^{(n)}$ be the universal $\cs$-algebra $\cs(\varphi, \psi \mid \hat{\mathcal{R}}_n)$, where $\hat{\mathcal{R}}_n$ denotes the set of relations:
\begin{enumerate}[(i)]
\item $\varphi$ and $\psi$ are c.p.c.\ order zero maps on $M_n$ and $M_2$ respectively;
\item $\psi(e_{11})=\varphi(1_n)(1-\varphi(1_n))$;
\item $\psi(e_{22})\varphi(e_{11})=\psi(e_{22})$.
\end{enumerate}
Then $W^{(n)} \cong W_{n,n+1}$.
\end{prop}

\begin{proof}
The proof is almost identical to that of Proposition~\ref{universal drop}, but we include it here for completeness. Define $\varphi:M_n \to W_{n,n+1}$ by
\[ 
\varphi(a)(t) = (a\otimes 1_n) \oplus (1-t)(a\otimes e_{n+1,n+1})
\]
for $a\in M_n$ and $t\in [0,1]$. Then $\varphi$ is clearly a c.p.c.\ order zero map. Equivalently, if we write
\[
x_i(t) = (e_{1i}\otimes 1_n) \oplus (1-t)^{1/2}(e_{1i}\otimes e_{n+1,n+1}) = \varphi^{1/2}(e_{1i})(t)
\]
for $1\le i \le n$, then the $x_i$ satisfy the order zero relations $\mathcal{R}_n^{(0)}$ and $\varphi(e_{ij})=x_i^*x_j$.
Next, define
\[
v(t) = t^{1/2}(1-t)^{1/2}\sum_{j=1}^n e_{j1}\otimes e_{n+1,j}.
\]
Then $vv^*= \varphi(1_n)(1-\varphi(1_n))$ and $vx_1=v$, and (so) $\|v\|\le 1$ and $v^2=0$. In particular, there is a unique c.p.c.\ order zero map $\psi: M_2 \to W_{n,n+1}$ with $\psi^{1/2}(e_{12})=v$, i.e.\
\[ 
\psi(e_{12})(t) = t(1-t)\sum_{j=1}^n e_{j1}\otimes e_{n+1,j},
\]
so that $\psi(e_{11})=vv^*$, $\psi(e_{22})=v^*v$ and $\varphi$ and $\psi$ satisfy all of the relations $\hat{\mathcal{R}}_n$.

Next, we check that $v$ and the $x_i$ generate $W_{n,n+1}$ as a $\cs$-algebra. Write $A:=\cs(\{v,x_1,\ldots,x_n\})$. We have
\[
v^*x_i(t) = t^{1/2}(1-t)(e_{1i}\otimes e_{1,n+1})
\]
and
\[
v^*x_ivx_j(t) = t(1-t)^{3/2}(e_{1j}\otimes e_{1i})
\]
for $1\le i,j \le n$. Thus, for $t\in(0,1)$, the elements $v^*x_i(t)$ and $v^*x_ivx_j(t)$ give all matrix units $\{e_{1k} \otimes e_{1l}\}_{1\le k \le n, 1\le l \le n+1}$, so generate all of $M_n\otimes M_{n+1}$, and so the irreducible representation $\ev_t:W_{n,n+1} \to M_n\otimes M_{n+1}$ restricts to an irreducible representation of $A$. For $t\in\{0,1\}$, the $x_i$ generate all the matrix units of $M_n$ in the endpoint irreducible representation $\ev_\infty:W_{n,n+1} \to M_n$. Thus every irreducible representation of $W_{n,n+1}$ restricts to an irreducible representation of $A$. Also, since $x_1(s)$ is not unitarily equivalent to $x_1(t)$ for distinct $s,t\in(0,1)$, it follows that inequivalent irreducible representations of $W_{n,n+1}$ restrict to inequivalent irreducible representations of $A$. Therefore, by Stone-Weierstrass (i.e.\ \cite[Proposition 11.1.6]{Dixmier:1964rt}), we do indeed have $\cs(\{v,x_1,\ldots,x_n\})=W_{n,n+1}$.\\

It remains to show that these generators of $W_{n,n+1}$ enjoy the appropriate universal property: for every representation $\{\hat\varphi, \hat\psi\}$ of the given relations, we need to show that there is a $^*$-homomorphism $W_{n,n+1} \to \cs(\hat\varphi, \hat\psi)$ sending $\varphi$ to $\hat\varphi$ and $\psi$ to $\hat\psi$. By \cite[Lemma 3.2.2]{Loring:1997it}, it suffices to consider the case where $\{\hat\varphi, \hat\psi\}$ is an \emph{irreducible} representation on some Hilbert space $H$ (i.e.\ has trivial commutant in $\mathfrak{B}(H)$). Note that the irreducible representations of $W_{n,n+1}$ are (up to unitary equivalence), the evaluation maps $\ev_t:W_{n,n+1} \to M_{n(n+1)}$ for $t\in (0,1)$ together with the endpoint representation $\ev_\infty:W_{n,n+1}\to M_n$. We will therefore show that (again, up to unitary equivalence) $\hat\varphi=\ev_t\circ\varphi$ and $\hat\psi=\ev_t\circ\psi$ for some $t\in(0,1)\cup\{\infty\}$.

For each $i\in\{1,\ldots,n\}$, let $\hat\psi_i:M_2 \to \cs(\hat\varphi, \hat\psi)$ be the c.p.c.\ order zero map defined by $\hat\psi_i^{1/2}(e_{12})=\hat\psi^{1/2}(e_{12})\hat\varphi^{1/2}(e_{1i})$, so that $\hat\psi_i(e_{11})=\hat\varphi(1_n)(1-\hat\varphi(1_n))$ and $\hat\psi_i(e_{22})\hat\varphi(e_{ii})=\hat\psi_i(e_{22})$. Define
\[
z:= \hat\psi(e_{11}) + \sum_{i=1}^n \hat\psi_i(e_{22}) \in \cs(\hat\varphi, \hat\psi).
\]
Then
\[
[z,\hat\varphi(e_{1j})] = \hat\psi_1(e_{22})\hat\varphi(e_{1j}) - \hat\varphi(e_{1j})\hat\psi_j(e_{22}) = 0, 
\]
and
\begin{align*}
[z,\hat\psi(e_{12})] &= \hat\psi^2(e_{12}) + \sum_{i=1}^n \hat\psi_i(e_{22})\hat\psi^{1/2}(e_{11})\hat\psi^{1/2}(e_{12}) - \sum_{i=1}^n \hat\psi(e_{12})\hat\varphi(e_{11})\hat\varphi(e_{ii})\hat\psi_i(e_{22})\\
&= \hat\psi^2(e_{12}) + 0 - \hat\psi^2(e_{12})\\
&= 0,
\end{align*}
so $z$ is central in $\cs(\hat\varphi, \hat\psi)$, and is therefore $\zeta 1$ for some scalar $\zeta$. Moreover, $z$ is positive with $\|z\|=\|\hat\psi(e_{11})\| = \|\hat\varphi(1_n)(1-\hat\varphi(1_n))\| \le 1/4$, so $0\le \zeta \le 1/4$.

If $\zeta = 0$ then $\hat\psi=0$ and $\hat\varphi(1_n)$ is therefore a projection. It follows that $\hat \varphi$ is a $^*$-homomorphism giving an irreducible representation of $M_n$ on $H$. Thus (up to unitary equivalence) $H=\mathbb{C}^n$ and $\hat\varphi=\ev_\infty\circ\varphi$.

Suppose that $\zeta>0$. Then $\zeta \hat\psi(e_{11}) = z \hat\psi(e_{11}) = (\hat\psi(e_{11}))^2$, so $p:=\zeta^{-1} \hat\psi(e_{11})$ and $q_i:=\zeta^{-1} \hat\psi_i(e_{22})$ are equivalent orthogonal projections with $p+q_1+\cdots+q_n=1$. Since $p$ commutes with $\hat\varphi(M_n)$, the maps $p\hat \varphi(\cdot)p$ and $(1-p)\hat \varphi(\cdot)(1-p)$ are c.p.c.\ order zero. In fact,
\[
\zeta \hat \varphi(1_n)(1-p) = \hat \varphi(1_n)(z- \hat\psi(e_{11}))
= z- \hat\psi(e_{11})
= \zeta (1-p),
\]
i.e.\ $(1-p)\hat \varphi(1_n)(1-p) = 1-p$. Thus, $(1-p)\hat \varphi(\cdot)(1-p)$ is a \emph{unital} c.p.c.\ order zero map into the corner $(1-p)\mathfrak{B}(H)(1-p) \cong \mathfrak{B}((1-p)H)$, so is a $^*$-homomorphism into this corner. Also, $p\hat \varphi(1_n)p$ commutes with (the WOT-closure of) the corner $p\cs(\hat\varphi, \hat\psi)p = p\cs(\hat\varphi)p$ (which, by irreducibility, is all of $p\mathfrak{B}(H)p \cong \mathfrak{B}(pH)$) so $p\hat \varphi(1_n)p=tp$ for some $t\in [0,1]$. So $t^{-1}p\hat\varphi(\cdot)p$ is also a $^*$-homomorphism, and is in fact an \emph{irreducible} representation of $M_n$ on $pH$. In particular, up to unitary equivalence, $pH=\mathbb{C}^n$ and $p\hat\varphi(\cdot)p=t\cdot\id_{M_n}$.

Moreover, since every $q_i$ is equivalent to $p$, they all have trace $n$ ($=\tr(p)$). Thus (again up to unitary equivalence) $(1-p)H = \mathbb{C}^{n^2}$ (so $H = \mathbb{C}^{n(n+1)}$) and $(1-p)\hat \varphi(\cdot)(1-p): M_n \to M_{n^2}$ is just $a \mapsto \diag(a,\ldots,a)$. Finally, since
\[
t(1-t)p = tp(p-tp) = \hat\varphi(1_n)p(p-\hat\varphi(1_n)p)=p\hat\varphi(1_n)(1-\hat\varphi(1_n))=p\hat \psi(e_{11}) = \zeta p,
\]
we have $t(1-t)=\zeta$. Therefore, $\hat\varphi=p\hat\varphi(\cdot)p+(1-p)\hat \varphi(\cdot)(1-p) = \ev_{1-t}\circ\varphi$ and, since $\zeta^{-1/2}\hat \psi^{1/2}(e_{12})$ is a partial isometry implementing an equivalence between $q_1$ and $p$, $\hat\psi = \ev_{1-t}\circ\psi$ (up to conjugation by a unitary). Thus $W_{n,n+1}$ has the required universal property.
\end{proof}

\begin{remark}
It should also be possible to detect $^*$-homomorphisms from $W_{n,n+1}$ to a stable rank one $\cs$-algebra $A$ at the level of the Cuntz semigroup $W(A)$ (just as for $Z_{n,n+1}$ in \cite[Proposition 5.1]{Rordam:2009qy}). The existence of $\langle x \rangle \in W(A)$ and a positive contraction $y\in A$ with $n\langle x \rangle = \langle y \rangle$ and $\langle y-y^2 \rangle \ll   \langle x \rangle$ (where $\ll$ denotes the relation of compact containment) is probably necessary and sufficient, but perhaps this is not the most useful characterization.
\end{remark}

Finally, we present $\mathcal{W}$ as a nonunital deformation of $\mathcal{Z}$.

\begin{thm} \label{main2}
Choose positive functions $d,f,g,h \in C_0(0,1]$, partial isometries $v_{k+1} \in M_{q(k+1)}$, and c.p.c.\ order zero maps $\rho_k: M_{q(k)} \to M_{q(k+1)}$ as in Theorem~\ref{main1}. Define $\mathcal{W}_U$ to be the universal $\cs$-algebra generated by c.p.c.\ order zero maps $\varphi_k$ on $M_{q(k)}$ ($k\in \mathbb{N}$) and $\psi_k$ on $M_2$ ($k\in \mathbb{N}$) such that for each $k$, these maps satisfy the relations $\hat{\mathcal{R}}_{q(k)}$, i.e.\
\begin{equation} \label{block1w}
\psi_k(e_{11})=\varphi_k(1_{q(k)})(1-\varphi_k(1_{q(k)}))
\end{equation}
and
\begin{equation} \label{block2w}
\psi_k(e_{22})\varphi_k(e_{11})=\psi_k(e_{22}),
\end{equation}
together with the additional relations $\hat{\mathcal{S}}_{q(k)}$ given by
\begin{equation} \label{connect1w}
\varphi_k = f(\varphi_{k+1}) \circ \rho_k,
\end{equation}
\begin{align} \label{connect2w}
\psi^{1/2}_k(e_{12}) &= f(\varphi_{k+1})(\rho_k(1_{q(k)}))^{1/2}\bigg(h(\varphi_{k+1})(1_{q(k+1)} - \rho_k(1_{q(k)}))^{1/2}f(\varphi_{k+1})(v_{k+1})\\
\notag &+ \left(1-f(\varphi_{k+1})(1_{q(k+1)}) + g(\varphi_{k+1})(1_{q(k+1)} - \rho_k(1_{q(k)}))\right)^{1/2} d(\psi_{k+1})(e_{12})\bigg).
\end{align}
Then $\mathcal{W}_U\cong \mathcal{W}$.
\end{thm}

\begin{proof}
The proof is essentially the same as that of Theorem~\ref{main1}, so we omit most of the details. As before, let us write $\hat\varphi_k=f(\varphi_{k+1}) \circ \rho_k$ and $\hat\psi^{1/2}_k(e_{12})= \gamma_k + \delta_k$, where this time 
\[
\gamma_k:= f(\varphi_{k+1})(\rho_k(1_{q(k)}))^{1/2} \lambda_k d(\psi_{k+1})(e_{12})
\]
and
\[
\delta_k:=f(\varphi_{k+1})(\rho_k(1_{q(k)}))^{1/2} \mu_k f(\varphi_{k+1})(v_{k+1}),
\]
with
\[
\lambda_k := (1-f(\varphi_{k+1})(1_{q(k+1)}) + g(\varphi_{k+1})(1_{q(k+1)} - \rho_k(1_{q(k)})))^{1/2}
\]
and
\[
\mu_k := h(\varphi_{k+1})(1_{q(k+1)} - \rho_k(1_{q(k)}))^{1/2}.
\]
To show that $\hat\psi_k(e_{11})=\hat\varphi_k(1_{q(k)})(1-\hat\varphi_k(1_{q(k)}))$, we proceed exactly as in the proof of Claim~\ref{block1 lemma}. The only difference is that we now have
\[
d(\psi_{k+1})(e_{11}) = d(\psi_{k+1}(e_{11})) = d(\varphi_{k+1}(1_{q(k+1)})(1-\varphi_{k+1}(1_{q(k+1)}))) = \hat d(\varphi_{k+1}(1_{q(k+1)})),
\]
where $\hat d(t) = d(t(1-t))$ as in (\ref{w dominance}). We also have
\[
f(\varphi_{k+1})(\rho_k(1_{q(k)}))(1-f(\varphi_{k+1})(1_{q(k+1)})) = \pi_{\varphi_{k+1}}(\rho_k(1_{q(k+1)}))(f-f^2)(\varphi_{k+1}(1_{q(k+1)})).
\]
Since $\hat d(f-f^2)=f-f^2$, this therefore gives
\[
f(\varphi_{k+1})(\rho_k(1_{q(k)}))(1-f(\varphi_{k+1})(1_{q(k+1)}))d(\psi_{k+1})(e_{11}) = f(\varphi_{k+1})(\rho_k(1_{q(k)}))(1-f(\varphi_{k+1})(1_{q(k+1)})),
\]
and the rest of the argument carries over mutatis mutandis. (Note in particular that $\lambda_k$ and $\mu_k$ both commute with $f(\varphi_{k+1})(\rho_k(1_{q(k)}))^{1/2}$.) The proof that $\hat\psi_k(e_{22})\hat\varphi_k(e_{11})=\hat\psi_k(e_{22})$ is literally the same as the proof of Claim~\ref{block2 lemma}.

We now know that $\mathcal{W}_U$ is isomorphic to an inductive limit $\varinjlim (W_{q(k),q(k+1)},\beta_k)$. Moreover, arguing exactly as before, we see that the connecting maps $\beta_k$ are (fibrewise) unitarily equivalent to the connecting maps $\alpha_k$ obtained earlier. That is, there are unitaries $z_k(t) \in M_{q(k+1)}\otimes M_{q(k+1)+1}$ such that
\begin{equation} \label{beta}
\beta_k^t = z_k(t)\left(\left(\bigoplus_{m=1}^{q(k+1)} \bigoplus_{i=1}^{q(k)-1} \ev_{\frac{i}{q(k)}}\right) \oplus \left(\bigoplus_{m=1}^{q(k)(q(k)-1)} \ev_{h(t)}\right) \oplus \left(\bigoplus_{i=1}^{q(k)} \ev_{h_i(t)}\right)\right)z_k(t)^*
\end{equation}
for every $t\in[0,1]$.

The same arguments as with $\mathcal{Z}_U$ show that $\mathcal{W}_U$ is simple and has a unique tracial state. (One has to perhaps be slightly careful about the \emph{existence} of a trace, since the space of tracial states of a nonunital $\cs$-algebra need not be compact. But this is not an issue.) The only minor technicality is that, since the building blocks $W_{q(k),q(k+1)}$ are nonunital and the connecting maps $\beta_k$ are degenerate,  $\mathcal{W}_U$ may have unbounded traces. However, one can easily show, using (\ref{proportion}), that this is not the case. It therefore follows from the classification theorem of \cite{Razak:2002kq} (or indeed from the more general result proved in \cite{Robert:2010qy}) that $\mathcal{W}_U \cong \mathcal{W}$.
\end{proof}

\begin{corollary} \label{canonical embedding}
There exists a trace-preserving embedding of $\mathcal{W}$ into $\mathcal{Z}$. Such an embedding is canonical at the level of the Cuntz semigroup, and is unique up to approximate unitary equivalence.
\end{corollary}

\begin{proof}
This follows immediately from Theorem~\ref{main2} and Theorem~\ref{main1}. The result can already be deduced from the main theorem of \cite{Robert:2010qy}, which also gives the uniqueness statement.
\end{proof}

\section{Outlook} \label{questions}

\subsection{} It might be interesting to characterize other $\cs$-algebras as we have done for $\mathcal{Z}$ and $\mathcal{W}$. It should in particular be possible, for any $n\ge 2$, to obtain a universal construction of a simple, monotracial, stably projectionless $\cs$-algebra $\mathcal{W}_n$ with $(K_0(\mathcal{W}_n),K_1(\mathcal{W}_n))=(0,\mathbb{Z}/(n-1)\mathbb{Z})$. Candidate building blocks could be of the form
\[
\{ f\in C([0,1], M_m\otimes M_{(n-1)(m+1)}) : f(0) = a\otimes 1_{(n-1)(m+1)}, f(1) = a\otimes 1_{(n-1)m},  \: a\in M_m \},
\]
which at least have the right $K$-theory. Of course, $\mathcal{W}_2$ is just $\mathcal{W}$, obtained as in Theorem~\ref{main2}.

It was proved in \cite{Robert:2010qy} that $\mathcal{W}\otimes \mathcal{K} \cong \mathcal{O}_2 \rtimes \mathbb{R}$ for certain `quasi-free' actions of $\mathbb{R}$ on the Cuntz algebra $\mathcal{O}_2$ (see for example \cite{Kishimoto:1996yu} and \cite{Dean:2001qy}). More generally, one would expect (i.e.\ the Elliott conjecture predicts) that $\mathcal{W}_n\otimes \mathcal{K} \cong \mathcal{O}_n \rtimes \mathbb{R}$, and in this sense $\mathcal{W}_n$ might be thought of as a stably projectionless analogue of $\mathcal{O}_n$. (Similar speculation is made in the article \cite{Nawata:2012fk}.)

It is unclear what interpretation the corresponding universal \emph{unital} algebras might have. Note for example that the Jiang-Su algebra is not stably isomorphic to a crossed product of a Kirchberg algebra by $\mathbb{R}$ (when simple, such a crossed product is either traceless or stably projectionless---see \cite[Proposition 4]{Kishimoto:1997kq}).

\subsection{} One of our motivations for presenting $\mathcal{Z}$ as a universal $\cs$-algebra was to find a direct proof of its strong self-absorption (i.e.\ one that does not rely on classification). To put this problem into context, consider the other strongly self-absorbing $\cs$-algebras. On the one hand, UHF algebras of infinite type can also be described in terms of order zero generators and relations, for example:
\[
M_{2^\infty} \cong \cs((\varphi_k)_{k=1}^\infty \mid \text{$\varphi_k$ order zero on $M_{q(k)}$}, \varphi_k(1_{q(k)})=\varphi_k(1_{q(k)})^2, \varphi_k = \varphi_{k+1}\circ\id_{q(k)}\otimes 1_{q(k)} \otimes 1_{q(k)})
\]
(where $q(k)$ is still $2^{3^k}$), and the proof of strong self-absorption in this case amounts to linear algebra. On the other hand, while $\mathcal{O}_2$ and $\mathcal{O}_\infty$ are presented simply as $\cs(s_1,s_2 \mid s_i^*s_i =1= s_1s_1^*+s_2s_2^*)$ and $\cs((s_i)_{i=1}^\infty \mid s_i^*s_j = \delta_{ij})$ respectively, the proofs that $\mathcal{O}_2 \otimes \mathcal{O}_2 \cong \mathcal{O}_2$ and $\mathcal{O}_\infty \otimes \mathcal{O}_\infty \cong \mathcal{O}_\infty$ require some difficult analysis (see for example \cite{Rordam:1994rm}). It is conceivable that our presentation of $\mathcal{Z}$ lies somewhere in the middle of this spectrum.

That being said, it is at least possible to show from our relations, in connection with \cite{Robert:2010qy}, that the $\cs$-algebra $\mathcal{Z}_U^{\otimes\infty}$ is strongly self-absorbing. (One can show that $\mathcal{Z}_U^{\otimes\infty}$ has stable rank one and strict comparison, and then use the main theorem of \cite{Robert:2010qy} to deduce strong self-absorption.)

Meanwhile, it remains an open problem to prove that $\mathcal{W}\otimes \mathcal{W} \cong \mathcal{W}$.


\begin{thebibliography}{10}

\bibitem{Dean:2001qy}
Andrew Dean.
\newblock A continuous field of projectionless {$\cs$}-algebras.
\newblock {\em Canad. J. Math.}, 53(1):51--72, 2001.

\bibitem{Dixmier:1964rt}
Jacques Dixmier.
\newblock {\em Les {$\cs$}-alg{\`e}bres et leurs repr{\'e}sentations}.
\newblock Cahiers Scientifiques, Fasc. XXIX. Gauthier-Villars \& Cie,
  {\'E}diteur-Imprimeur, Paris, 1964.

\bibitem{Jacelon:2010fj}
Bhishan Jacelon.
\newblock A simple, monotracial, stably projectionless {$\cs$}-algebra.
\newblock {\em \emph{arXiv preprint math.OA/1006.5397; to appear in} J. London
  Math. Soc.}, 2010.

\bibitem{Jiang:1999hb}
Xinhui Jiang and Hongbing Su.
\newblock On a simple unital projectionless {$\cs$}-algebra.
\newblock {\em Amer. J. Math.}, 121(2):359--413, 1999.

\bibitem{Kishimoto:1996yu}
Akitaka Kishimoto and Alex Kumjian.
\newblock Simple stably projectionless {$\cs$}-algebras arising as crossed
  products.
\newblock {\em Canad. J. Math.}, 48(5):980--996, 1996.

\bibitem{Kishimoto:1997kq}
\bysame.
\newblock Crossed products of {C}untz algebras by quasi-free automorphisms.
\newblock In {\em Operator algebras and their applications ({W}aterloo, {ON},
  1994/1995)}, volume~13 of {\em Fields Inst. Commun.}, pages 173--192. Amer.
  Math. Soc., Providence, RI, 1997.

\bibitem{Loring:1997it}
Terry~A. Loring.
\newblock {\em Lifting solutions to perturbing problems in {$\cs$}-algebras},
  volume~8 of {\em Fields Institute Monographs}.
\newblock American Mathematical Society, Providence, RI, 1997.

\bibitem{Matui:2012qv}
Hiroki Matui and Yasuhiko Sato.
\newblock Strict comparison and {$\mathcal{Z}$}-absorption of nuclear
  {$\cs$}-algebras.
\newblock {\em \emph{arXiv preprint math.OA/1111.1637; to appear in} Acta
  Math.}, 2011.

\bibitem{Nawata:2012fk}
Norio Nawata.
\newblock Picard groups of certain stably projectionless {$\cs$}-algebras.
\newblock {\em \emph{arXiv preprint math.OA/1207.1930}}, 2012.

\bibitem{Razak:2002kq}
Shaloub Razak.
\newblock On the classification of simple stably projectionless {$\cs$}-algebras.
\newblock {\em Canad. J. Math.}, 54(1):138--224, 2002.

\bibitem{Robert:2010qy}
Leonel Robert.
\newblock Classification of inductive limits of 1-dimensional {NCCW} complexes.
\newblock {\em \emph{arXiv preprint math.OA/1007.1964; to appear in} ÊAdv. Math.}, 2010.

\bibitem{Robert:2011fj}
\bysame.
\newblock The cone of functionals on the {C}untz semigroup.
\newblock {\em \emph{arXiv preprint math.OA/1102.1451; to appear in} Math. Scand.}, 2011.

\bibitem{Rordam:1994rm}
Mikael R{\o}rdam.
\newblock A short proof of {E}lliott's theorem:
  {$\mathcal{O}_2\otimes\mathcal{O}_2\cong\mathcal{O}_2$}.
\newblock {\em C. R. Math. Rep. Acad. Sci. Canada}, 16(1):31--36, 1994.

\bibitem{Rordam:2009qy}
Mikael R{\o}rdam and Wilhelm Winter.
\newblock The {J}iang-{S}u algebra revisited.
\newblock {\em J. Reine Angew. Math.}, 642:129--155, 2010.

\bibitem{Sato:2010rz}
Yasuhiko Sato.
\newblock The {R}ohlin property for automorphisms of the {J}iang-{S}u algebra.
\newblock {\em J. Funct. Anal.}, 259(2):453--476, 2010.

\bibitem{Toms:2007uq}
Andrew~S. Toms and Wilhelm Winter.
\newblock Strongly self-absorbing {$\cs$}-algebras.
\newblock {\em Trans. Amer. Math. Soc.}, 359(8):3999--4029 (electronic), 2007.

\bibitem{Tsang:2005fj}
Kin-Wai Tsang.
\newblock On the positive tracial cones of simple stably projectionless {$\cs$}-algebras.
\newblock {\em J. Funct. Anal.}, 227(1):188--199, 2005.

\bibitem{Winter:2007qf}
Wilhelm Winter.
\newblock Localizing the {E}lliott conjecture at strongly self-absorbing {$\cs$}-algebras.
\newblock {\em \emph{arXiv preprint math.OA/0708.0283; to appear in} J. Reine
  Angew. Math.}, 2007.

\bibitem{Winter:2010hl}
\bysame.
\newblock Decomposition rank and {$\mathcal{Z}$}-stability.
\newblock {\em Invent. Math.}, 179(2):229--301, 2010.

\bibitem{Winter:2009yq}
\bysame.
\newblock Strongly self-absorbing {$\cs$}-algebras are {$\mathcal{Z}$}-stable.
\newblock {\em J. Noncommut. Geom.}, 5(2):253--264, 2011.

\bibitem{Winter:2012pi}
\bysame.
\newblock Nuclear dimension and {$\mathcal{Z}$}-stability of pure
  {$\cs$}-algebras.
\newblock {\em Invent. Math.}, 187(2):259--342, 2012.

\bibitem{Winter:2009sf}
Wilhelm Winter and Joachim Zacharias.
\newblock Completely positive maps of order zero.
\newblock {\em M{\"u}nster J. Math.}, 2:311--324, 2009.

\end{thebibliography}
\end{document}